\documentclass[12pt]{elsarticle}

\usepackage{amssymb}
\usepackage{amsthm}

\usepackage{amsfonts,amsmath}

\theoremstyle{remark}
\newtheorem{theorem}{Theorem}
\newtheorem{lemma}{Lemma}
\newtheorem{proposition}{Proposition}

\newtheorem{remark}{Remark}
\newtheorem{example}{Example}

\newtheorem{corollary}{Corollary}

\begin{document}

\begin{frontmatter}

\title{A Flat Triangular Form for Nonlinear Systems with Two Inputs: Necessary
and Sufficient Conditions\tnoteref{sup}}

\tnotetext[sup]{The first author was fully supported by CAPES and
FUNPESQUISA/UFSC. The second author was partially supported by CNPq. The
third author was partially supported by ``Agence Nationale de la Recherche''
(ANR),  Projet Blanc EMAQS number ANR-2011-BS01-017.}

\author[UFSC]{H.B.\ Silveira\corref{cor}}
\ead{hector.silveira@ufsc.br}
\address[UFSC]{Departamento de Automa\c{c}\~ao e Sistemas (DAS),
Federal University of Santa Catarina (UFSC), Florian\'opolis, Brazil}
\cortext[cor]{Corresponding author.}

\author[USP]{P.S.\ Pereira da Silva}
\ead{paulo@lac.usp.br}
\address[USP]{Escola Polit\'{e}cnica -- PTC,
University of S\~ao Paulo (USP), S\~ao
Paulo, Brazil}

\author[MPT]{P.\ Rouchon}
\ead{pierre.rouchon@mines-paristech.fr}
\address[MPT]{Centre Automatique et Syst\`emes (CAS), Mines ParisTech, Paris, France}

\date{~}
\begin{abstract}
The present work establishes necessary and sufficient conditions for a
nonlinear system with two inputs to be described by a specific triangular form.
Except for some regularity conditions, such triangular form is flat.
This may lead to the discovery of new flat
systems. The proof relies on well-known results for driftless
systems with two controls (the chained form) and on geometric tools from
exterior differential systems. The paper also illustrates the application of
its results on an academic example and on a reduced order
model of an induction motor.
\end{abstract}


\end{frontmatter}


\section{Introduction}
Consider the control system with two inputs
\begin{equation}\label{cs}
    \dot{x}=f(x) + g_1(x)u_1 + g_2(x)u_2,
\end{equation}
where $x \in U \subset \mathbb{R}^n$ is the state, $U$ is open, $n
\geq 2$,
$u=(u_1,u_2) \in \mathbb{R}^2$ is the control and $f, g_1, g_2$:~$U
\rightarrow
\mathbb{R}^n$ are smooth (i.e.\ infinitely differentiable) mappings.
Recall that \eqref{cs} is driftless when $f=0$ and that $u=\alpha(x)
+ \beta(x) v$ is a regular static state feedback
defined on an open set $V \subset U$ if the mappings
$\alpha=(\alpha_1,\alpha_2)$:~$V \rightarrow \mathbb{R}^2$ and
$\beta=(\beta_{ij})$:~$V \rightarrow \mathbb{R}^{2 \times 2}$ are smooth and
the matrix
$\beta(x)=(\beta_{ij}(x)) \in \mathbb{R}^{2 \times 2}$ is invertible, for
all $x \in V$, where $v=(v_1,v_2) \in \mathbb{R}^2$ is the new control.
It was established in \cite{Mur95} necessary and sufficient conditions
for \eqref{cs} with $f=0$ to be (locally) described around a given $x_0 \in U$
by the chained form
\begin{equation}\label{tf-dlt}
  \begin{split}
  \dot{z}_1 &= z_2 v_1, \\
  & \;\; \vdots   \\
  \dot{z}_{n-2} &= z_{n-1}v_1, \\
  \dot{z}_{n-1} &= v_2, \\
  \dot{z}_n &= v_1,
  \end{split}
\end{equation}
after a change of coordinates $z=\varphi(x)$
and a regular feedback $u=\beta v$ (see Theorem~\ref{cfec} in
Section~\ref{pr-tf}). In order to take into account the drift $f$ in
(\ref{cs}), this work adds some geometric conditions (concerning exterior differential
systems) to the ones of \cite{Mur95} and presents necessary and sufficient
conditions
for (\ref{cs}) (with or without drift) to be described around a given $x_0 \in
U$ by the triangular form
\begin{equation}\label{tf-e}
    \begin{split}
    \dot{z}_1&=\phi_1(z_1,z_2,z_n)+z_2v_1, \\
    \dot{z}_2&=\phi_2(z_1,z_2,z_3,z_n)+z_3 v_1, \\
             & \;\; \vdots \\
    \dot{z}_{n-3}&=\phi_{n-3}(z_1,z_2,\dots,z_{n-2},z_n)+z_{n-2}v_1, \\
    \dot{z}_{n-2}&=\phi_{n-2}(z_1,z_2, \dots,z_n)+z_{n-1}v_1, \\
    \dot{z}_{n-1}&= v_2, \\
    \dot{z}_n&=v_1,
    \end{split}
\end{equation}
after a change of coordinates $z=\varphi(x)$ and a regular static state feedback
 $u=\alpha+ \beta v$.
Such result is the main contribution of this paper (see Theorem~\ref{tf-nsc} in Section~\ref{nsc-tf}). Note that the control vector fields in \eqref{tf-dlt} and \eqref{tf-e} are exactly the same, and that the drift $f$ is taken into account by means of condition~2) of Theorem~\ref{tf-nsc}.
Furthermore, \cite{MarRou94} also considers (\ref{cs}) with $f=0$ and exhibits sufficient conditions for the chained form description
\eqref{tf-dlt} to hold generically, that is, on an open
dense set (see Theorem~\ref{tf-dls}). This work then
establishes as an immediate corollary of the main result sufficient geometric conditions for
\eqref{cs} to be generically described by the triangular form \eqref{tf-e} (see
Corollary~\ref{tf-dsc}).

There are many motivations for choosing the specific triangular form (\ref{tf-e}). Firstly, \eqref{tf-e} is (essentially) flat. This is discussed in detail in the subsequent paragraphs. Secondly, (\ref{tf-e}) may be seen as a generalization of the chained form (\ref{tf-dlt}) to
systems with drift. Although there are many possible such generalizations, the conditions obtained in this paper with respect to the triangular form \eqref{tf-e} may be regarded as generalizations of the ones of  \cite{Mur95} and \cite{MarRou94} for the chained form (\ref{tf-dlt}), since new conditions have been added to them in order to consider the drift $f$ in (\ref{cs}).
Thirdly, the control literature has recently demonstrated interest in the triangular form (\ref{tf-e}) (see e.g.\ \cite{LiChaoSuChu13}). To the best of the authors' knowledge, the first work in the literature that considered and characterized the triangular form (\ref{tf-e}) was the thesis \cite[Teorema~3.10, p.\ 51]{Silveira10} of one of the authors. The present paper slightly improves the results of \cite{Silveira10} and illustrates the constructive aspects of the results in the examples. Lastly, \cite{BouBouBarKra11} provides necessary and sufficient conditions for a multi-input nonlinear system to be described, after a change of coordinates, by a flat triangular canonical form which differs from (\ref{tf-e}). In Example~1 in Section~\ref{nsc-tf}, one exhibits a nonlinear system of the form \eqref{cs} which is transformed into (\ref{tf-e}) by means of the application of the results here established, but which fails to meet the referred conditions of \cite{BouBouBarKra11}.

It is well-known (and immediate to verify) that $y=(z_1, z_n)$ is a
flat\footnote{For the concept of flatness, see e.g.\
\cite{MarMurRou03},
\cite{Lev09}, \cite{SirAgr04}.} output for the chained form (\ref{tf-dlt}) around the points in which $v_1
\neq 0$. For the triangular form \eqref{tf-e}, straightforward computations using the
implicit function theorem show that $y=(z_1, z_n)$ is a flat output
around the points that satisfy $v_1 + \partial \phi_1 / \partial z_2(z) \neq 0,
\dots, v_1 + \partial \phi_{n-2} / \partial z_{n-1}(z) \neq 0$.
Hence, except for these regularity conditions, this paper establishes
sufficient conditions for \eqref{cs}
to be flat, which therefore may lead to the discovery of new flat systems. The corresponding constructive procedure for describing \eqref{cs} by the form \eqref{tf-e}, and hence constructing a flat output, is illustrated in Example~2 at the end of Section~\ref{nsc-tf}. It may be summarized as follows (see Remark~\ref{construction} for further details). Assume that both conditions in Theorem~\ref{tf-nsc} are met. First, one constructs a change of coordinates $z=\varphi(x)$ and $\beta(x)$ in the regular feedback $u=\beta v$ such that the control vector fields $g_1, g_2$ in \eqref{cs} are described by the chained form \eqref{tf-dlt} (there are available methods for this in the literature, and the existence of such $z=\varphi(x)$ and $u=\beta v$ is ensured by condition 1) of Theorem~\ref{tf-nsc}). Then, condition 2) of Theorem~\ref{tf-nsc} implies that after applying the regular feedback $u=\alpha + \beta v$ in \eqref{cs}, where $\alpha = \beta \overline{\alpha}$ and $\overline{\alpha}= -( \langle dz_n, f \rangle, \langle dz_{n-1}, f \rangle)$, one has that the closed-loop drift $f+\alpha_1 g_1+\alpha_2 g_2$ exhibits the triangular structure in \eqref{tf-e} in the coordinates $z=\varphi(x)$. Hence,
\eqref{cs} has been transformed into \eqref{tf-e}. Lastly, $y=(z_1, z_n)$ is the resulting flat output constructed.

One recalls that the practical importance of flat systems is that the control problems of motion planning, trajectory tracking and flat
output tracking are easily (and sometimes trivially) treated
\cite[]{MarMurRou03, Lev09, SirAgr04}. It is still an open problem in the control literature how to determine if a general nonlinear system is flat and, in such case, how a flat output can be constructed. As is well-known, a system which is linearizable by regular static state feedback is flat, but the converse is false. The results here
presented could establish that a system of the form \eqref{cs} is flat even when it
is not linearizable by regular static feedback.
This is the case, for instance, when the distribution determined by the control vector fields $g_1, g_2$
in \eqref{cs} is not involutive, so that \eqref{cs} is not static feedback linearizable (see e.g.\ \cite[Proposition 9.16]{Sas99}).
Nonetheless, such system could still meet the conditions in Theorem~\ref{tf-nsc} (or Corollary~\ref{tf-dsc}) of the present work, in which case $y=(z_1,z_n)$ would be a flat output.
Therefore, the results here established may lead to the discovery of new flat systems. As a final remark, it is clear that not all flat systems are static feedback equivalent to the triangular form (3). Indeed, if (1) is static feedback linearizable, then the distribution spanned by $g_1, g_2$ must be involutive, and hence condition 1) of Theorem 3) is violated.

The rest of the paper is organized as follows. Section~\ref{mcn} sets up the
notation and presents the specific distributions and
concepts from exterior differential systems (associated and retracting spaces)
that are required in subsequent sections. The well-known results of \cite{Mur95} and
\cite{MarRou94} for driftless systems
that were mentioned above are
given in Section~\ref{pr-tf}. Section~\ref{nsc-tf} exhibits the proof of the
main result of this work and its corollary,
along with two examples illustrating its application.

\section{Mathematical Preliminaries}\label{mcn}

The sets of natural and real numbers are denoted as $\mathbb{N}$ (where $0 \in \mathbb{N}$) and $\mathbb{R}$, respectively. Throughout the text, unless otherwise stated, $M$ is a
finite-dimensional
smooth manifold
and $U \subset M$ is an open set. Then,
$\mathcal{C}^{\infty}(U)$, $\mathfrak{X}(U)$ and $\Omega^{r}(U)$ denote the set
of smooth functions on $U$, the set of
smooth vector fields on $U$ and the set of $r$-forms on $U$, respectively, for
all $r \in \mathbb{N}$. Given $p \in M$, $T_p M$ and $T_p^*M$ are the tangent
and
cotangent spaces to $M$
at $p$, respectively, and $\Omega_p^{r} M$ is the set of
alternating tensors of covariant  order $r$ on $T_p M$.
The set $\Omega_p M=\bigoplus_{r=0}^{\infty} \Omega_p^{r}M$
(direct sum) is the exterior algebra over $T_p M$, where
$\Omega_p^{0}M=\mathbb{R}$, and
$\Omega(U)=\bigoplus_{r=0}^{\infty} \Omega^{r}(U)$ is  the exterior algebra on
$U$,  where $\Omega^{0}(U)=\mathcal{C}^{\infty}(U)$. Let $f \in
\mathcal{C}^{\infty}(U)$, $X, Y \in
\mathfrak{X}(U)$, $\omega \in \Omega(U)$. Then: $L_{X} f \in
\mathcal{C}^{\infty}(U)$ denotes the Lie derivative of $f$
with respect to $X$, $[X,Y]
\in \mathfrak{X}(U)$ is the Lie bracket of $X$ and $Y$, $L_{X}
\omega \in \Omega(U)$ is the Lie derivative of $\omega$ with respect to $X$,
and $d\omega \in
\Omega(U)$ is the exterior derivative of $\omega$.
Let $X^1, \dots, X^r \in
\mathfrak{X}(U)$ and $V \subset U$. Given $q \in V$, $\{ X^1,
\dots, X^r \}|_q$ denotes $\{X^1_q, \dots,
X^r_q\}$, and when $X^1_p, \dots,
X^r_p$ are independent vectors for all $p \in V$, one simply says
that $X^1, \dots, X^r$ are independent \emph{on} $V$. Similar conventions
shall also be adopted for $1$-forms.

Following \cite{Sas99}, one considers in this work a
distribution (respectively, codistribution) on $U$ simply as a
submodule of $\mathfrak{X}(U)$ (resp., $\Omega^1(U)$).
Let $\Delta$ be a
distribution on $U$ and $p \in U$.
Define $\Delta_p=\{X_p \in T_p M
\; | \; X_p=Y_p, \mbox{ where } Y \in \Delta
\}$. The concept of a regular point and the
annihilator $\Delta^\perp \subset \Omega^{1}(U)$ of $\Delta$ are defined in the
usual manner. One writes $\dim(\Delta)=r$ when $\dim (\Delta_p)=r$, for all $p
\in U$.
Given an open set $V \subset U$,
define the distribution
$\Delta|V=\mbox{span}_{\mathcal{C}^{\infty}(V)}\{ X \in
\mathfrak{X}(V) \; | \; X=Y|V, \mbox{ where } Y \in \Delta \}$.
Note that $\{ X \in \mathfrak{X}(V) \; | \; X=Y|V, \mbox{ with
} Y \in \Delta \}$ is not necessarily a distribution on $V$.
It is clear that analogous concepts to the ones presented above for
distributions can be readily defined for codistributions.
Let $\Delta_1, \Delta_2, \Delta_3$ be distributions on $U$. Consider the
distributions $\Delta_1 + \Delta_2 = \{ X \in \mathfrak{X}(U) \; | \;
X=X^1+X^2, \mbox{ where } X^1 \in \Delta_1, X^2 \in \Delta_2 \}$ and
$[ \Delta_1 , \Delta_2 ] = \mbox{span}_{\mathcal{C}^{\infty}(U)}\{ X \in
\mathfrak{X}(U) \; | \;
    X=[X^1,X^2], \mbox{ where } X^1 \in \Delta_1, X^2 \in \Delta_2 \}$.
Note that $\Delta_1 +
[ \Delta_2 , \Delta_3 ] = \mbox{span}_{\mathcal{C}^{\infty}(U)}\{
X \in \mathfrak{X}(U) \; | \; X=X^1+[X^2,X^3], \mbox{ where } X^i \in
\Delta_i, 1 \leq i \leq 3 \}$. If $V \subset U$ is open, then $(\Delta_2 +
\Delta_3)|V  = \Delta_2|V + \Delta_3|V$ and
$(\Delta_1+[\Delta_2,\Delta_3])|V  = \Delta_1|V + [\Delta_2|V,\Delta_3|V]$.

Let $\Delta$ be a distribution on $U$.
Define, for each $k \in \mathbb{N}$, the distributions
$F_0 = G_0 = \Delta$, $F_{k+1} = F_k + [F_k, F_0] \supset F_k$, $G_{k+1} =
G_k + [G_k , G_k] \supset G_k \supset F_k$.
The sequences of distributions $(F_k)$ and $(G_k)$ are called the \emph{Lie
flag} and the \emph{derived flag} of $\Delta$, respectively \cite{Res01}.

The next result, which is asserted in \cite{Mur95} and \cite{Sas99}, follows by
induction and the usual properties of the Lie
bracket.

\begin{proposition}\label{clbd-lb}
Let $(F_k)$ and $(G_k)$ be the Lie and derived
flags of $\Delta=\mbox{span}_{\mathcal{C}^\infty(U)}\{X^1,
X^2\}$, respectively, where $X^1, X^2 \in
\mathfrak{X}(U)$. Take $P_0 = Q_0 = \{ X^1, X^2 \}$ and consider, for each
$k \in \mathbb{N}$,
\begin{align*}
    & P_{k+1} = \{ X \in \mathfrak{X}(U) \; | \;
X=[Y^{k+2},[Y^{k+1},[\dots,[Y^2,Y^1]
\dots ]], \\ & \hspace{106pt} \mbox{ where } Y^1, \dots, Y^{k+2} \in P_0
\}, \\
    & Q_{k+1} = \{ X \in \mathfrak{X}(U) \; | \; X=[Y^1,Y^2],
    \mbox{ where } Y^1, Y^2 \in \cup_{j=0}^k Q_j \} \subset Q_{k+2}.
\end{align*}
Then, $F_0 = G_0 =
\Delta$ and, for each $k \in \mathbb{N}$,
\begin{align*}
    F_{k+1} &= \mbox{span}_{\mathcal{C}^\infty(U)}(\cup_{j=0}^{k+1} P_{j}) =
F_k + \mbox{span}_{\mathcal{C}^\infty(U)}(P_{k+1}), \\
    G_{k+1} &= G_0 + \mbox{span}_{\mathcal{C}^\infty(U)}(Q_{k+1}) = G_k +
\mbox{span}_{\mathcal{C}^\infty(U)}(Q_{k+1}).
\end{align*}
\end{proposition}

\begin{remark}\label{flags}
It is clear that $F_0=G_0=\mbox{span}\{X^1,X^2\}$, $F_1=G_1=\mbox{span}\{X^1,X^2,$\linebreak$[X^1,X^2]\}$, $F_2=G_2=\mbox{span}\{X^1,X^2,[X^1,X^2],[X^1,[X^1,X^2]],[X^2,[X^1,X^2]]\}$. However, $F_k$ and $G_k$ are not necessarily identical when $k \geq 3$.
\end{remark}

\begin{remark}\label{lbd-gp-r}
Consider system \eqref{cs} and take $X^1=g_1, X^2=g_2$. Let
$(F_k)$ and $(G_k)$ be the Lie and derived flags of
$\Delta=\mbox{span}_{\mathcal{C}^\infty(U)}\{g_1,
g_2\}$, respectively.
It is immediately verified by induction that $(F_k)$  and $(G_k)$ are invariant under regular static state feedbacks. More precisely,
suppose
that $u=\alpha + \beta v$ is such a feedback defined on
$U$. Let $\widehat{\Delta}=\mbox{span}_{\mathcal{C}^{\infty}(U)}\{\widehat{g}_1,
\widehat{g}_2\}$, where $\widehat{g}_i=\beta_{i1} g_1 + \beta_{i2}g_2$, $i=1,2$,
and let
$(\widehat{F}_k)$ and $(\widehat{G}_k)$ be the Lie and derived flags of
$\widehat{\Delta}$, respectively. Then,
$F_k=\widehat{F}_k$, $G_k=\widehat{G}_k$, for $k \in
\mathbb{N}$.
\end{remark}

Let $p \in M$ and $X_p \in T_p M$. The \emph{interior product}
with
$X_p$ is the map $i_{X_p}$:~$\Omega_{p}^{r} M \rightarrow
\Omega_{p}^{r-1} M$ defined as
$(i_{X_p} \omega_p)(Y_p^1, \dots, Y_p^{r-1})=\omega_p(X_p, Y_p^1, \dots,
Y_p^{r-1})$, for all $\omega_p \in \Omega_{p}^{r} M$ and
all $Y_p^1, \dots, Y_p^{r-1} \in T_p M$, where $r \in \mathbb{N}$,
$\Omega_{p}^{-1} M \triangleq \{ 0 \}$ and
$i_{X_p} \omega_p \triangleq 0$, for all $\omega_p \in \Omega_{p}^{0}
M=\mathbb{R}$. Extending it by linearity, one obtains the
mapping $i_{X_p}$:~$\Omega_{p} M \rightarrow \Omega_{p} M$.
One also denotes $i_{X_p}\omega_p$ as $X_p \, \lrcorner
\, \omega_p$. The map $i_{X_p}$ is linear, satisfies $i_{X_p} \circ
i_{X_p}=0$ and is an antiderivation, that is, if $\theta_p
\in \Omega_{p}^{r}M$ and $\sigma_p \in \Omega_{p}^{s}M$, then
$X_p \, \lrcorner \, (\theta_p \wedge \sigma_p)=(X_p \, \lrcorner \,
\theta_p) \wedge \sigma_p + (-1)^{r}\theta_p \wedge (X_p \, \lrcorner \,
\sigma_p)$ \cite{Sas99}. For the purposes of this work, it
suffices to consider the definition
of associated and retracting
spaces given in the sequel. For a more general notion (but equivalent in
the present context) concerning
(homogeneous) algebraic ideals, see \cite{Yang92}, \cite{Sas99}, \cite{Dieu74-4}.
Let $\Lambda=\mbox{span}_{\mathcal{C}^{\infty}(U)} \{
\lambda^1, \dots, \lambda^m \}$ be a codistribution on $U$. One defines the
\emph{associated space} or \emph{Cauchy
characteristic space} determined by $\Lambda$ at $p \in U$ as
\begin{align*}
    A(\Lambda)_p=\{ X_p \in T_p M \; | &  \; X_p \, \lrcorner \, d \lambda^i_p
\in \Lambda_p=\mbox{span}_{\mathbb{R}} \{ \lambda^1_p, \dots, \lambda^m_p \} \\
& \mbox{ and } \langle \lambda^i_p, X_p \rangle = 0, \mbox{ for
} 1 \leq i \leq m \},
\end{align*}
and the \emph{retracting space} determined by $\Lambda$ at $p \in U$
as $C(\Lambda)_p=(A(\Lambda)_p)^{\perp}$.
Note that $A(\Lambda)_p$ and $C(\Lambda)_p$
are subspaces of $T_p M$ and $T_p^{*} M$, respectively.
Furthermore, given an open set $V
\subset U$ and $p \in V$, one has $
A(\Lambda)_p=A(\Lambda|V)_p$.

\section{Known Results for Driftless Systems}\label{pr-tf}

Theorem~\ref{cfec} and Theorem~\ref{tf-dls} given below are the well-known
results for driftless systems with two inputs mentioned
in the introduction. They will be used in the results established in the
next section. Note that Theorem~\ref{tf-dls} is a generic result.

\begin{theorem}\cite{Mur95}\label{cfec}
Consider system (\ref{cs}) with $f=0$.
Define $\Delta=\mbox{span}_{\mathcal{C}^{\infty}(U)}\{g_1,
g_2\}$ and $\Lambda = \Delta^\perp$. Let $(F_k)$ and $(G_k)$ be the Lie and
derived flags of $\Delta$, respectively. Given $p \in U$, the following
assertions are equivalent:
\begin{enumerate}[1)]
    \item For each $0 \leq k \leq n-2$, $p \in U$ is a regular point of $F_k,
G_k$ with $\dim(F_k(p)) = \dim(G_k(p)) = 2 + k$;
    \item There exists a change of coordinates $z=\varphi(x)$ and a
regular static state feedback $u=\beta v$, both of them defined on
an open neighborhood $V \subset U$ of $p$, in which the expression of the
resulting closed-loop system
in the coordinates $z=(z_1, \dots, z_n)$ is given by \eqref{tf-dlt}.
\end{enumerate}
\end{theorem}

\begin{theorem}\cite{MarRou94}\label{tf-dls}
Consider system (\ref{cs}) with $f=0$.
Let $(G_k)$ be the derived flag of
$\Delta=\mbox{span}_{\mathcal{C}^{\infty}(U)}\{g_1, g_2\}$
and suppose that $\dim(G_k(q)) = 2 + k$, for all $0 \leq k \leq n-2$, $q \in D$,
where $D \subset U$ is an open dense set in $U$. Then, there exists an open
and dense set $V$ in $U$ such that, for every $p \in V$, property 2) above
holds.
In particular, for every $p \in V \subset U$, property 1) is true.
\end{theorem}

\section{Main Result}\label{nsc-tf}

The proof of Theorem~\ref{tf-nsc}, which is the main result of this paper, relies on the lemmas given below.

\begin{lemma}\label{rs-cb}
Let $M$ be a manifold with $\dim(M) = n \geq 4$.
Assume that $z=(z_1, \dots, z_n)$ is a local chart of
$M$ defined on an open set $U$. Consider the following codistributions on
$U$
\begin{align*}
    \Lambda^{k}=\mbox{span}_{\mathcal{C}^\infty(U)} \{dz_1-z_2dz_n, \dots,
dz_{n-2-k}-z_{n-1-k}dz_n \},
\end{align*}
for each $1 \leq k \leq n-3$. Then, $C^k_{q}=\mbox{span}_{\mathbb{R}} \{
dz_1, \dots, dz_{n-1-k}, dz_n
\}|_{q} \supset \Lambda^{k}_q$,
for all $q \in U$, $1 \leq k \leq n-3$, where
$C^k_{q} \triangleq C(\Lambda^k)_{q}$
is the retracting space determined by $\Lambda^k$ at
$q \in U$.
\end{lemma}
\begin{proof}
Let $q \in U$, $1 \leq k \leq n-3$ and $A^k_{q}  =
A(\Lambda^k)_{q}$. Then $S_q = \{ \partial / \partial
z_{n-1-k}, \dots,$ $\partial / \partial z_{n-1},
\gamma \}|_q$ is a basis of $(\Lambda^k_q)^\perp$, where $\gamma = z_2 \partial
/ \partial
z_{1} + \dots + z_{n-1-k} \partial /
\partial z_{n-2-k} + \partial / \partial z_{n}$.
Let $X_q \in T_{q}(M)$. Then, $X_q \in A^k_q$ if and only
if $X_q \in \mbox{span}_{\mathbb{R}} (S_q)$ and $X_q \, \lrcorner \, (dz_j
\wedge dz_n)|_q \in \Lambda^k_q$ for $2 \leq j \leq n-1-k$, which in turn is
equivalent to $X_q \in \mbox{span}_{\mathbb{R}} \{\partial /
\partial z_{n-k}, \dots, \partial /
\partial z_{n-1} \}|_q$. This completes the proof.
\end{proof}

\begin{lemma}\label{lbd-cb}
Let $U \subset M$ be open with $\dim(M) = n \geq 3$. Consider that $z=(z_1, \dots, z_n)$ is a local chart defined on $U$ such that
$X^1, X^2 \in \mathfrak{X}(U)$ are respectively described as $X^1(z) = (z_2, \dots, z_{n-1}, 0 , 1)$, $X^2(z) = (0 , \dots, 0, 1, 0)$
(note that $X^2=\partial / \partial z_{n-1}$). Let $(F_k)$ and $(G_k)$ be the Lie and derived flags of $\Delta=\mbox{span}_{\mathcal{C}^\infty(U)} \{ X^1, X^2 \}$, respectively.
Then, $F_{k} = G_{k} = \mbox{span}_{\mathcal{C}^\infty(U)} \{ X^1, \partial / \partial z_{n-1},$ \linebreak $\dots, \partial / \partial z_{n-1-k} \}$  and $\dim(F_k) = \dim(G_k) = 2 + k$, for $0 \leq k \leq n-2$.
\end{lemma}
\begin{proof}
One begins by showing by induction on $0 \leq k \leq n-2$ that
\begin{align}\label{lbd-cb-e}
    G_{k} = \mbox{span}_{\mathcal{C}^\infty(U)} \{ X^1, \partial / \partial z_{n-1},  \dots, \partial / \partial z_{n-1-k} \}.
\end{align}
Note that $G_{0} = \mbox{span}_{\mathcal{C}^\infty(U)} \{ X^1, \partial / \partial z_{n-1} \}$. Consider the induction hypothesis: $G_{k} = \mbox{span}_{\mathcal{C}^\infty(U)} \{ X^1, \partial / \partial z_{n-1}, \dots, \partial / \partial z_{n-1-k} \}$, where $0 \leq k \leq n-3$. Define $H = G_k + \mbox{span}_{\mathcal{C}^\infty(U)}\{\partial / \partial z_{n-1-(k+1)}\}$. It suffices to prove that $G_{k+1} = H$. By hypothesis, $X^1(z) = (z_2,  \dots, z_{n-1}, 0 , 1)$.
It is easy to see that\footnote{See also \cite[Lemma 11.8]{Sas99}.}
\begin{equation}\label{lbs}
    [\partial / \partial z_{n-j}, X^1] = \partial / \partial z_{n-1-j} \in F_j \subset G_j, \quad  \mbox{for } 1 \leq j \leq n-2.   
\end{equation}
In particular, $\partial / \partial z_{n-1-(k+1)} = \partial / \partial z_{n-2-k} \in G_{k+1}$. Hence, $H \subset G_{k+1}$ because $G_k \subset G_{k+1}$. Now, let $Y = Y^1 + [Y^2, Y^3] \in G_{k+1}$, where $Y^i \in G_k$, for $1 \leq i \leq 3$. One has $Y^i = \alpha_i X^1 + \sum_{j=1}^{k+1} \beta_{ij} \partial / \partial z_{n-j}$, where
$\alpha_i, \beta_{ij} \in \mathcal{C}^{\infty}(U)$, for $1 \leq i \leq 3$. It follows from (\ref{lbs}) that
$[Y^2, Y^3] = [\alpha_1 X^1 + \sum_{j=1}^{k+1} \beta_{1j} \partial / \partial z_{n-j}$, $\alpha_2 X^1 + \sum_{j=1}^{k+1} \beta_{2j} \partial / \partial z_{n-j}] = \lambda X^1 + \sum_{j=1}^{k+1} \gamma_j \partial / \partial z_{n-j} + \sum_{j=1}^{k+1} \delta_j \partial / \partial z_{n-1-j} \in H$, where
$\lambda, \gamma_j, \delta_j \in \mathcal{C}^{\infty}(U)$. Since $Y^1 \in G_k \subset H$, one has $Y \in H$. Thus, $G_{k+1} \subset H$, and hence $G_{k+1}=H$.
Finally, recall that $F_k \subset G_k$, for $0 \leq k \leq n-2$, with $F_0 = G_0=\Delta$. On the other hand, it is immediate from (\ref{lbd-cb-e}) and (\ref{lbs}) that $G_k \subset F_k$, for $1 \leq k \leq n-2$. Therefore, $F_k = G_k$, $0 \leq k \leq n-2$.
\end{proof}

\begin{theorem}\label{tf-nsc}
Consider system (\ref{cs}) and let $(F_k)$ and $(G_k)$ be the Lie and
derived flags of $\Delta =
\mbox{span}_{\mathcal{C}^\infty(U)} \{ g_1,
g_2 \}$, respectively. For each $1 \leq k \leq n-3$, define
$\Lambda^{k}=(G_k)^{\perp}$ and $C^k_{q}=C(\Lambda^k)_{q}$, for $q \in
U$. Let $p \in U$.
Then, there exist a change of coordinates $z=\varphi(x)$
and a regular static state feedback of the form
$u=\alpha + \beta v$, both of them defined on an open neighborhood $V
\subset U$ of $p$, such that the expression of the resulting closed-loop
system in the coordinates $z=(z_1, \dots,
z_n)$ is given by the triangular form \eqref{tf-e} \emph{if and only if} there
exists an open neighborhood $W \subset U$ of $p$ in which the following
geometric conditions are satisfied:
\begin{enumerate}[1)]
    \item $\dim(F_k(q)) = \dim(G_k(q)) = 2 + k$, for $0 \leq k \leq
n-2$,
$q \in W$;
    \item $L_{f}\omega_{q} \in C^k_{q}$, for $1 \leq k \leq n-3$, $q \in W$, $\omega \in \Lambda^k$.
\end{enumerate}
\end{theorem}

\begin{remark}\label{construction}
Note that the control vector fields in \eqref{tf-dlt} and \eqref{tf-e} are exactly the same, that condition~1) above is precisely assertion~1)
in Theorem~\ref{cfec} for driftless systems, and that the drift $f$ in (\ref{cs}) is taken into account by means of condition~2). Assume that both conditions of Theorem~\ref{tf-nsc} are met. It can be seen from the sufficiency part of the proof of Theorem~\ref{tf-nsc} that the required change of
coordinates $z=\varphi(x)$ and $\beta(x)$ in the regular feedback $u=\alpha + \beta v$ may always be taken as the ones provided by Theorem~\ref{cfec} (except for a possible restriction of their domain of definition). This means that, in order for \eqref{cs} to be described by \eqref{tf-e}, it suffices to obtain a change of coordinates $z=\varphi(x)$ and a feedback $\beta(x)$ such that the controls vector fields $g_1, g_2$ in \eqref{cs} are described by the chained form \eqref{tf-dlt} with $z=\varphi(x)$ and $u=\beta v$. The existence of such $z=\varphi(x)$ and $u=\beta v$ is ensured by condition 1) of Theorem~\ref{tf-nsc}, and in general they may be constructed by the techniques given in \cite[Algorithm~1]{TilMurSas95} and \cite[Section~3]{Res01}. Another approach is the result in \cite[Proposition~7]{MurSas93}. Having obtained the referred $z=\varphi(x)$ and $u=\beta v$, the proof of Theorem~\ref{tf-nsc} shows that condition 2) ensures that after applying the regular feedback $u=\alpha + \beta v$ in \eqref{cs}, where $\alpha = \beta \overline{\alpha}$ and $\overline{\alpha}= -( \langle dz_n, f \rangle, \langle dz_{n-1}, f \rangle)$ (cf.\ \eqref{fd-l} and \eqref{feedback}), one has that the closed-loop drift $f+\alpha_1 g_1+\alpha_2 g_2$ exhibits the triangular structure in \eqref{tf-e} in the coordinates $z=\varphi(x)$. Thus,
\eqref{cs} has been transformed into \eqref{tf-e}.
\end{remark}

Before proving Theorem~\ref{tf-nsc}, one explains the conventions adopted in its
proof.
Let $\Delta$ a distribution on $U \subset M$ and $X^1, \dots, X^r \in
\mathfrak{X}(U)$. If
$\Delta=\mbox{span}_{\mathcal{C}^{\infty}(U)}\{ X^1, \dots, X^r \}$,
then one simply writes $\Delta=\mbox{span}\{ X^1,
\dots, X^r \}$. Let $V,W \subset U$ be open sets with $V \subset
W \subset U$ and $Y^1, \dots, Y^m \in \mathfrak{X}(W)$.
In case $\Delta|V=\mbox{span}_{\mathcal{C}^{\infty}(V)}\{ Y^1|V, \dots, Y^m|V
\}$, then one writes $\Delta|V=\mbox{span}\{ Y^1, \dots,$ $Y^m \}$
\emph{over} ${\mathcal{C}^{\infty}(V)}$. Analogous notations for
codistributions will also be used.

\begin{proof}
The result is shown for $n \geq 4$.
The reader will have no difficulty in verifying that the arguments presented
also assure its validity for $2 \leq n \leq 3$. One begins by
proving necessity.

\noindent \emph{(Necessity)} Consider the resulting closed-loop vector fields on $W=V$:
\begin{equation}\label{acfr}
    \begin{array}{l}
        \widehat{f} = f + \alpha_1 g_1 + \alpha_2 g_2 \in \mathfrak{X}(W), \\
	\widehat{g}_1 = \beta_{11} g_1 + \beta_{12} g_2, \quad \widehat{g}_2 = \beta_{21} g_1 + \beta_{22} g_2 \in G_0|W,
    \end{array}
\end{equation}
with $\beta=(\beta_{ij})$. The expressions of
$\widehat{f}, \widehat{g}_1, \widehat{g}_2$ in the coordinates
$z=(z_1,\dots,z_n)$ are respectively given as
\begin{equation}\label{ha-cd}
    \begin{array}{l}
    \widehat{f}(z)=(\phi_1(z), \dots, \phi_{n-2}(z),0,0), \\
    \widehat{g}_1(z)=(z_2,\dots,z_{n-1},0,1), \\
    \widehat{g}_2(z)=(0, \dots, 0, 1, 0),
    \end{array}
\end{equation}
where $\phi_1, \dots, \phi_{n-2}$ are as in \eqref{tf-e}.
Let $(\widehat{F}_k)$ and $(\widehat{G}_k)$ be the Lie and the derived
flags of $\widehat{\Delta}=\mbox{span}_{\mathcal{C}^\infty(W)} \{
\widehat{g}_1, \widehat{g}_2 \}$, respectively.
According to Lemma~\ref{lbd-cb}, $\widehat{F}_{k} = \widehat{G}_{k} =
\mbox{span}_{\mathcal{C}^\infty(W)} \{
\widehat{g}_1, \partial / \partial z_{n-1},$ $\partial / \partial z_{n-2},
\dots,
\partial / \partial z_{n-1-k} \}$ with
$\dim(\widehat{F}_k) = \dim(\widehat{G}_k) = 2 +
k$, for $0 \leq k
\leq n-2$, and $\widehat{g}_2=\partial / \partial z_{n-1}$. Since
$\beta=(\beta_{ij})$ is invertible on
$W$, one concludes from Remark~\ref{lbd-gp-r} that $F_{k}|W = \widehat{F}_{k} =
\widehat{G}_{k} = G_{k}|W$ with $\dim(F_{k}|W) =
\dim(G_{k}|W) = 2 + k$, for every $0 \leq k \leq
n-2$. Define $\theta_i = dz_i-z_{i+1}dz_n$, for $1 \leq i \leq n-3$. Fix $1 \leq i \leq n-3$. By restricting $W$ if necessary, one has that
\begin{equation}\label{awd-cd}
    \begin{split}
    \Lambda^{n-2-i}|W = (G_{n-2-i}|W)^{\perp} = \mbox{span}
    \{ \theta_1,  \dots, \theta_i \} \subset (G_0|W)^{\perp}
    \end{split}
\end{equation}
over $\mathcal{C}^\infty(W)$. Therefore, using Lemma~\ref{rs-cb},
\begin{equation}\label{rs-cd}
    C^{n-2-i}_q=\mbox{span}_{\mathbb{R}}
    \{dz_1, \dots, dz_{i+1}, dz_n \}|_q \supset \Lambda^{n-2-i}_q, \quad \mbox{for } q \in W.
\end{equation}

From (\ref{ha-cd}), one gets $L_{\widehat{f}} \theta_i=
        d\phi_i - \phi_{i+1}dz_n$,
$L_{\widehat{g}_1} \theta_i =
        dz_{i+1} - z_{i+2} dz_n$,
$L_{\widehat{g}_2} \theta_i = 0$,
with $d\phi_i \in \mbox{span}_{\mathcal{C}^\infty(W)} \{ dz_1,
dz_2,$ $\dots, dz_{i+1}, dz_n \}$. By (\ref{rs-cd}),
$(L_{\widehat{f}} \theta_i)_{q}, (L_{\widehat{g}_1}
\theta_i)_{q},$ $(L_{\widehat{g}_2} \theta_i)_{q} \in
C^{n-2-i}_{q}$,
for $q \in W$. Let $q \in W$,
$\omega \in \Lambda^{n-2-i}$, $\overline{\omega} = \omega|W$. Then (\ref{awd-cd}) and
(\ref{rs-cd}) imply
$L_{\widehat{f}}\overline{\omega}_{q}$,
$L_{\widehat{g}_1}\overline{\omega}_{q}$, $L_{\widehat{g}_2}\overline{
\omega}_{q} \in C^{n-2-i}_{q}$, and (\ref{acfr}) gives that
$L_{f}\overline{\omega}_{q} \in
C^{n-2-i}_{q}$ since $\langle \overline{\omega}, \widehat{g}_1
\rangle = \langle \overline{\omega}, \widehat{g}_2 \rangle = 0$ by
\eqref{awd-cd} again.

\noindent \emph{(Sufficiency)} By hypothesis, $p$ is a regular point of
$F_k, G_k$ with $\dim(F_k(p)) = \dim(G_k(p)) = 2 + k$, for $0 \leq k
\leq n-2$.
Hence Theorem~\ref{cfec} establishes that, by restricting $W$ if necessary,
there exists a change of coordinates $z=\varphi(x)$ and regular static state
feedback
$u=\beta v$, both defined on $W$, such that the expressions of $f$,
\begin{equation}\label{racfr}
    \begin{array}{l}
        \widehat{g}_1=\beta_{11} g_1 +
\beta_{12} g_2, \quad \widehat{g}_2=\beta_{21} g_1 + \beta_{22}
g_2,
    \end{array}
\end{equation}
in the coordinates $z=(z_1, \dots, z_n)$, are respectively given as
\begin{equation}\label{ha-cd-l}
    \begin{array}{l}
     f(z)=(\gamma_1(z), \dots, \gamma_n(z)), \\
    \widehat{g}_1(z) = (z_2, \dots, z_{n-1}, 0, 1),
\\  \widehat{g}_2(z) = (0, \dots, 0, 1, 0).
    \end{array}
\end{equation}
It can be assumed
that\footnote{This follows by using an adequate bump function and by restricting
$W$ if necessary.}
$\beta=\widetilde{\beta}|W$, where
$\widetilde{\beta}=(\widetilde{\beta}_{ij})$:~$U \rightarrow \mathbb{R}^{2 \times
2}$ is smooth.

Since $\beta=(\beta_{ij})$ is invertible on $W$, the
proof of necessity above provides, by restricting $W$ if necessary,  that
\begin{align}\label{lbd-cb-l}
    \begin{array}{l}
    G_{k}|W =\mbox{span} \{ \widehat{g}_1, \partial /
\partial z_{n-1}, \dots, \partial / \partial
z_{n-1-k} \}
    \end{array}
\end{align}
with $\widehat{g}_2=\partial /
\partial z_{n-1}$, $\dim(G_k|W) = 2 + k$, for $1 \leq k \leq
n-2$, and
\begin{equation}\label{las}
    \Lambda^{n-2-i}|W = (G_{n-2-i} | W)^{\perp} = \mbox{span} \{\theta_1, \dots,
\theta_{i}\}
\end{equation}
over $\mathcal{C}^\infty(W)$, for $1 \leq i \leq n-3$, where $\theta_i = dz_i-z_{i+1}dz_n$.
Hence,
(\ref{las}) and Lemma~\ref{rs-cb}
give
\begin{equation}\label{rs-cd-l}
    C^{n-2-i}_{q} = \mbox{span}_{\mathbb{R}} \{dz_1, \dots, dz_{i+1},
dz_n \}|_{q}
\supset \Lambda^{n-2-i}_{q},  \, q \in W, \; 1 \leq i \leq n-3.
\end{equation}

Consider the following maps on $W$ (see
(\ref{ha-cd-l})):
\begin{equation}\label{fd-l}
    \begin{split}
    & \overline{\alpha} = (\overline{\alpha}_1, \overline{\alpha}_2) \triangleq
-( \langle dz_n, f \rangle, \langle dz_{n-1}, f \rangle)  = -(\gamma_{n},
\gamma_{n-1}), \\
    & \widehat{f} =  f + \overline{\alpha}_1 \widehat{g}_1
+ \overline{\alpha}_2 \widehat{g}_2, \\
    & \phi_i = \langle dz_{i}, \widehat{f} \rangle, \qquad \mbox{for } 1 \leq i \leq n-2.
    \end{split}
\end{equation}
By construction,
\begin{equation}\label{fdv-cd-l}
    \begin{split}
    & \langle
dz_{n}, \widehat{f} \rangle = \langle
dz_{n-1}, \widehat{f} \rangle = 0.
    \end{split}
\end{equation}
Indeed, $\langle dz_{n}, \widehat{g}_1
\rangle = \langle dz_{n-1}, \widehat{g}_2 \rangle = 1, \langle dz_{n-1},
\widehat{g}_1 \rangle = \langle dz_{n}, \widehat{g}_2 \rangle = 0$ (cf.\ \eqref{ha-cd-l}). Thus,
$\langle dz_{n-j}, \widehat{f} \rangle = \gamma_{n-j} - \gamma_{n-j} = 0$, for $0 \leq j
\leq 1$.

Now, because $\widehat{f}(z) = (\phi_1(z), \dots, \phi_{n-2}(z),0,0)$, it remains to show that
\begin{equation}\label{zdc}
  d \phi_i|_{q} \in
\mbox{span}_{\mathbb{R}} \{ dz_1, \dots,  dz_{i+1}, dz_n \}|_{q}, \quad \mbox{for } q \in W, \; 1 \leq i \leq n-3.
\end{equation}
Fix $1 \leq i \leq n-3$, $q \in W$. By \eqref{ha-cd-l}, $L_{\widehat{g}_1} \theta_i =
 dz_{i+1} - z_{i+2} dz_n$ and $L_{\widehat{g}_2} \theta_i = 0$, and $L_{\widehat{f}}\theta_i =
d\phi_i - \phi_{i+1} dz_n$ from (\ref{fdv-cd-l}).
It can be assumed, by
restricting
$W$ if necessary, that
$\overline{\alpha}=\widetilde{\alpha}|W$ and $\theta_i = \widetilde{\theta_i}|W$, where
$\widetilde{\alpha}=(\widetilde{\alpha}_1,\widetilde{\alpha}_2)$,
$\widetilde{\alpha}_1, \widetilde{\alpha}_2 \in
\mathcal{C}^{\infty}(U)$, and $\widetilde{\theta_i} \in \Lambda^{n-2-i} = (G_{n-2-i})^\perp$.
Since $\beta = \widetilde{\beta} | W$, $G_0
\subset G_k$ ($k \in \mathbb{N}$) and $\widetilde{\theta}_i \in \Lambda^{n-2-i}$, it follows
from condition 2) in Theorem~\ref{tf-nsc}, \eqref{racfr}, \eqref{rs-cd-l} and (\ref{fd-l}) that
$(L_{\widehat{f}}\theta_i)_{q} = (L_{\widehat{f}}\widetilde{\theta}_i)_{q}  \in
C^{n-2-i}_{q}$. Hence, \eqref{zdc} holds.

Finally, let
$v=(v_1, v_2) \in \mathbb{R}^2$ and define $X^v=\widehat{f} + \widehat{g}_1 v_1 + \widehat{g}_2
v_2 \in
\mathfrak{X}(W)$. By (\ref{ha-cd-l}), (\ref{fd-l}),
(\ref{fdv-cd-l}) and (\ref{zdc}), the expression of $X^v$ in the
coordinates $z=(z_1,\dots,z_n)$ on $W$ is given by
\eqref{tf-e},
by restricting $W$ to an open connected set if necessary. Therefore,
\begin{equation}\label{feedback}
	u = \alpha + \beta v = \beta \overline{\alpha} + \beta v, \qquad \mbox{with } \overline{\alpha} = -( \langle dz_n, f \rangle, \langle dz_{n-1}, f \rangle),
\end{equation}
is the desired regular
feedback, because $\widehat{f}= f + \overline{\alpha}_1 \widehat{g}_1
+ \overline{\alpha}_2 \widehat{g}_2 $, $\widehat{g}_1 = \beta_{11}
g_1 + \beta_{12} g_2$, $\widehat{g}_2 = \beta_{21} g_1 + \beta_{22} g_2$,
$\overline{\alpha}=(\overline{\alpha}_1, \overline{\alpha}_2)$ and
$\beta=(\beta_{ij})$.
\end{proof}

By using Theorem~\ref{tf-dls}, one obtains the generic result:

\begin{corollary}\label{tf-dsc}
Consider system (\ref{cs}) and define $G_k$, $\Lambda^{k}$,
$C^k_{q}$ as in
Theorem~\ref{tf-nsc}. Suppose that there exists an open dense set $D$ in $U$
such that:
\begin{enumerate}[1)]
    \item $\dim(G_k(q)) = 2 + k$, for $0 \leq k \leq n-2$, $q \in D$;
    \item $L_{f}\omega_{q} \in C^k_{q}$, for $1 \leq k \leq n-3$, $q
\in D$, $\omega \in \Lambda^{k}$.
\end{enumerate}
Then, there exists an open dense set $V$ in $U$ such that, for all $p
\in V$, there exists a change of coordinates $z=\varphi(x)$
and a regular static state feedback of the form
$u=\alpha + \beta v$, both defined on an open neighborhood $W
\subset U$ of $p$, in which the expression of the resulting closed-loop system
in the coordinates $z=(z_1, \dots, z_n)$ is given by
\eqref{tf-e}.
\end{corollary}

One exhibits in the sequel two examples of application of Theorem~\ref{tf-nsc}.

\begin{example}\label{tfsc-e}
Consider system \eqref{cs} with $U=\mathbb{R}^4$, $x=(x_1, \dots, x_4)
\in
\mathbb{R}^4$ and
\begin{align*}
 & f(x)=(0,x_1^2 + x_2,1,x_1 x_4), \\
 & g_1(x)=(x_4^2+1,(x_3 -2 x_1)(x_4^2+1), 0, (x_1^2+x_2)(x_4^2+1)), \\
 & g_2(x)=(0,0,1,0).
\end{align*}
By relying on Remark~\ref{construction}, it will be shown that such system can be described by
(\ref{tf-e}) around every point of $\mathbb{R}^4$. One has
\begin{align*}
  & g_3(x)=[g_1, g_2](x)=(0, -(x_4^2+1), 0, 0), \\
  & g_4(x)=[g_1, g_3](x)=(0,-2x_4(x_1^2+x_2)(x_4^2+1),0,(x_4^2+1)^2), \\
  & g_5(x) = [g_2, g_3](x) = 0,
\end{align*}
for $x \in \mathbb{R}^4$.
According to Remark~\ref{flags},
$\dim(F_k)=\dim(G_k) = 2+k$, for $0 \leq k \leq 2$. Since
$G_1=\mbox{span}_{\mathcal{C}^\infty(\mathbb{R}^4)}\{g_1, g_2, g_3\}$, it is straightforward to obtain\footnote{A computational procedure for determining
$C^1_{p}$ is outlined in the proof of Lemma~\ref{rs-cb}.}
\begin{align*}
  &\Lambda^1 = (G_1)^\perp
    = \mbox{span}_{\mathcal{C}^\infty(\mathbb{R}^4)}\{\gamma\}, \quad \mbox{where } \gamma=(x_1^2+x_2) dx_1 - dx_4, \\
  & C^1_{p} = \mbox{span}_{\mathbb{R}}\{dx_1, dx_2, dx_4\}|_{p} \supset
\Lambda^1_p,
\end{align*}
for $p \in \mathbb{R}^4$. It is easy
to verify that $L_{f} \gamma_{p}=[(x_1^2 + x_2 - x_4)dx_1 - x_1 dx_4]|_p \in C^1_{p}$,
for $p \in \mathbb{R}^4$. Using $\Lambda^1_p \subset C^1_{p}$, one
concludes that
$L_{f} \omega_{p} \in
C^1_{p}$, for $\omega \in \Lambda^1$, $p \in \mathbb{R}^4$, that is, one has
shown that the conditions in Theorem~\ref{tf-nsc} are met on $\mathbb{R}^4$.
The (global) change of coordinates $z=\varphi(x)=(x_4,x_1^2+x_2,x_3,x_1)$
transforms $f(x), g_1(x), g_2(x)$ into
$f(z)=(z_1z_4,z_2,1,0)$, $g_1(z)=(z_2(z_1^2+1),z_3(z_1^2+1),0,z_1^2+1)$,
$g_2(z)=(0,0,1,0)$.
The obvious choice for $\beta$ in the feedback $u=\beta v$ is then
$\beta=(\beta_{ij}) =
\mbox{diag}((z_1^2+1)^{-1},1)=\mbox{diag}((x_4^2+1)^{-1},1) \in
\mathbb{R}^{2\times2}$, since
it yields $\widehat{g}_1(z)=(z_1^2+1)^{-1} g_1(z)=(z_2,z_3,0,1)$,
$\widehat{g}_2(z)=g_2(z)=(0,0,1,0)$.
It remains to specify $\alpha$ in $u=\alpha + \beta v$. Following Remark~\ref{construction}, one takes
$\alpha=\beta \overline{\alpha}$, where $\overline{\alpha}
=-(\langle dz_4,f \rangle,\langle dz_{3},f \rangle)=(0,-1)$.

Hence, after applying the
(global) feedback $u=\alpha + \beta v$, the closed-loop system in the
coordinates $z=(z_1,z_2,z_3,z_4)$ has the form \eqref{tf-e} with
\begin{align*}
 \dot{z} = (z_1z_4,z_2,0,0) + (z_2,z_3,0,1)v_1 + (0,0,1,0)v_2.
\end{align*}
Note that $y=(z_1,z_4)=(x_4,x_1)$ is a flat output around the points in which
$v_1 \neq 0$.

Although the present academic example has codimension $n-2=2$,
it does not satisfy the sufficient geometric conditions for flatness established in
\cite[p.\ 266]{BouBouBarKra11}. Indeed, $[g_1,g_2] \notin \mbox{span}\{g_1, g_2
\}$, $[g_1,g_2] \notin \mbox{span}\{g_1, \mbox{ad}_f g_1\}$, $[g_1,g_2] \notin
\mbox{span}\{g_2, \mbox{ad}_f g_2\}$,
for all $x \in \mathbb{R}^4$. Therefore, one concludes that the present example cannot be transformed into the flat triangular
canonical form considered in \cite{BouBouBarKra11} after a change of coordinates.

\end{example}

\begin{example}
As a practical engineering example,
consider the reduced order model (current-fed) of an induction motor with three
states
\begin{align*}
 \dot{\omega} & = \dfrac{n_p M}{JL}(\psi_a i_b - \psi_b i_a) - \dfrac{T_L}{J},
\bigskip \\
 \dot{\psi}_a & = - \dfrac{R}{L} \psi_a - n_p \omega \psi_b + M \dfrac{R}{L}
i_a, \bigskip \\
 \dot{\psi}_b & = - \dfrac{R}{L} \psi_b + n_p \omega \psi_a + M \dfrac{R}{L}
i_b,
\end{align*}
where $x=(x_1,x_2,x_3)=(\omega,\psi_a, \psi_b) \in \mathbb{R}^3$ is the state, $u=(u_1,u_2)=(i_a,
i_b) \in \mathbb{R}^2$ is the control, and $J, L, M, n_p, R, T_L \in \mathbb{R}$ are
the (constant) parameters. See \cite[p.\ 276]{MarTom95} for details. This system has the form $\eqref{cs}$ and it meets the conditions of Theorem~\ref{tf-nsc} at all $x
\in \mathbb{R}^3$. In fact, $\dim(F_k)=\dim(G_k) = 2+k$, for $0 \leq k \leq 1$ (see Remark~\ref{flags}). By Remark~\ref{construction}, in order to describe this system by the  triangular form \eqref{tf-e}, it suffices to find a change of coordinates $z=\varphi(x)$ and a feedback $\beta(x)$ in which the corresponding control vector fields $g_1, g_2$ are described by the chained form  \eqref{tf-dlt} with $z=\varphi(x)$ and $u = \beta v$. In the previous academic example they were directly provided, but in this one they shall be constructed based on \cite[Proposition~7]{MurSas93}. This will allow the construction of a flat output accordingly.

Using the same notation as in \cite[Proposition~7]{MurSas93}, it suffices to find smooth functions $h_1, h_2$ such that
$L_{g_1} h_1 = 1$, $dh_1 \Delta_1 = dh_2 \Delta_2=0$, where $\Delta_1 = \{g_2, [g_1, g_2] \}$ and $\Delta_2=\{g_2\}$. An obvious choice is
$h_1(x) = (MR)^{-1} L x_2$ and $h_2(x) = L^{-1} MR x_1 - (JL)^{-1} n_p M  x_2 x_3$. Consequently, \cite[Proposition~7]{MurSas93} implies that $z_1 = h_2(x)$, $z_2 = L_{g_1} h_2 (x) = -2 (J L^2)^{-1} n_p M  R^2  x_3$, $z_3 = h_1(x)$ is the desired change of coordinates and that
\[
	\beta =
\left(
\begin{array}{cc}
	1 & 0 \\
	L_{g_1}^2 h_1   & L_{g_2} L_{g_1} h_2

\end{array}
\right)^{-1} =
\left(\begin{array}{cc} 1 & 0\\ 0 & - (2 n_p M^3  R^2 )^{-1} J L^3  \end{array}\right)
\]
is the required (constant) feedback matrix. Such change of coordinates is in fact global, that is, it is a diffeomorphism from $\mathbb{R}^3$ onto $\mathbb{R}^3$, since $L_{g_2} L_{g_1} h_2(x)\neq0$ for all $x \in \mathbb{R}^3$. One referes the reader to the proof of \cite[Proposition~7]{MurSas93} for details.

Now, following Remark~\ref{construction}, take
\[
	\alpha = - \beta (\langle dz_3,f \rangle,\langle dz_{2},f \rangle)=(M R)^{-1} (R x_2 + n_p L  x_1 x_3, R x_3 - n_p L  x_1 x_2).
\]
After applying the global feedback $u=\alpha + \beta v$, one obtains that the closed-loop system is described in the global coordinates $z=(z_1,z_2,z_3)=\varphi(x)$ by the
triangular form \eqref{tf-e}:
\[
	\dot{z} = (\phi_1(z),0,0) + (z_2,0,1)v_1 + (0,1,0)v_2,
\]
where
\[
\phi_1(z) = \frac{-2 J^2 L^6 z_1 z_2^2 - 8 n_p^2 M^6  R^4  z_1 z_3^2 + 4 n_p^2 M^6 R^4  z_2 z_3^3 + J^2  L^6  z_2^3 z_3 - 8  L  M^5 R^4 T_L}{8 J L^2 M^4 R^3}.
\]
According to the regularity conditions stated in the introduction, one concludes that $y=(z_1,z_3)$ is a flat output for the induction motor model above around the points that satisfy
\[
	v_1 + \partial \phi_1 / \partial z_2(z) = v_1 - \frac{J L^4 z_2 \left(4 z_1 - 3 z_2 z_3\right)}{8 M^4 R^3} +  \frac{n_p^2 M^2 R  z_3^3}{2 J L^2}  \neq 0
\]
(in the $z$ coordinates), or equivalently, since $u_1=v_1$,
\[
	u_1 + \frac{n_p L  \left(n_p x_2^3 + 2 J R x_1 x_3 + n_p x_2  x_3^2  \right)}{2 J M R^2} \neq 0
\]
(in the original coordinates). Note that $z=\varphi(x)$ and $u=\alpha+\beta v$, as well the flatness condition above, do not depend on the torque load $T_L$.
\end{example}

\section{Conclusion}
This work has established necessary and sufficient geometric conditions for
\eqref{cs} to be described by the triangular form \eqref{tf-e}, which is (essentially) flat. One has treated
systems with only two inputs ($m=2$). A generalization of the conditions here obtained to systems with $m > 2$ inputs
has been recently achieved in \cite{Nicolau14}. In terms of future research, the triangular form
\eqref{tf-e} may give rise to the development of constructive
steering methods for systems with drift. The techniques described in
\cite{Sas99} and \cite{TilMurSas95} for the chained form \eqref{tf-dlt}
could point to that direction.

\def\cprime{$'$}

\end{document}